\documentclass{article}
\usepackage{amsmath,amssymb}
\usepackage{amscd}
\usepackage{comment}
\usepackage{color}
\usepackage{enumerate}
\usepackage{cases}
\usepackage{mathtools}
\usepackage[all]{xy}
\usepackage{amsthm}
\numberwithin{equation}{section}
\theoremstyle{definition}
\newtheorem{defi}{Definition}
\newtheorem{theo}{Theorem}
\newtheorem{lemm}{Lemma}
\newtheorem{prop}{Proposition}
\newtheorem{cor}{Corollary}
\newtheorem{ex}{Example}
\newtheorem{rem}{Remark}

\newcommand{\sslash}{{/\mkern-6mu/}}
\def\ad{{\rm ad}}
\def\Aut{{\rm Aut}}

\def\U{{\rm U}}
\def\GL{{\rm GL}}

\def\Exp{{\rm Exp}}

\def\Exp{{\rm Exp}}
\def\Im{{\rm Im}}

\def\Sym{{\rm Sym}}

\def\Vol{{\rm Vol}}

\def\R{{\mathbb R}}
\def\Z{{\mathbb Z}}
\def\C{{\mathbb C}}
\def\N{{\mathbb N}}
\def\Q{{\mathbb Q}}

\def\g{\mathfrak{g}}

\def\h{\mathfrak{h}}
\def\s{\mathfrak{s}}

\def\M{{\cal M}}

\def\inum{{\sqrt{-1}}}

\def\vol{{d\mu_{g_M}}}

\def\gM{{g_M}}

\begin{document}
\title {Generalized Kazdan-Warner equations associated with a linear action of a torus on  a complex vector space}
\author {Natsuo Miyatake}
\date{}
\maketitle
\begin{abstract}
We introduce generalized Kazdan-Warner equations on Riemannian manifolds associated with a linear action of a torus on a complex vector space. We show the existence and the uniqueness of the solution of the equation on any compact Riemannian manifold. As an application, we give a new proof of a theorem of Baraglia \cite{Bar1} which asserts that a cyclic Higgs bundle gives a solution of  the periodic Toda equation.
\end{abstract}

\section{Introduction}
\label{intro}
Let $K$ be a closed connected subtorus of a real torus $T^d:=\U(1)^d$ with the Lie algebra $k\subseteq t^d$. We denote by $\iota^\ast: (t^d)^\ast \rightarrow k^\ast$ the dual map of the inclusion map $\iota: k\rightarrow t^d$. Let $u_1,\dots, u_d$ be a basis of $t^d$ defined by
\begin{align*}
u_1\coloneqq&(\inum, 0, \dots, 0), \\
u_2\coloneqq&(0, \inum, 0,\dots, 0), \\
&\cdots \\
u_d\coloneqq&(0,\dots, 0,\inum).
\end{align*}
 We denote by $u^1, \dots, u^d\in (t^d)^\ast$ the dual basis of $u_1, \dots, u_d$. Let $(\cdot, \cdot)$ be the metric on $t^d$ and $(t^d)^\ast$ satisfying
\begin{align*}
(u_i, u_j)=(u^i, u^j)=\delta_{ij} \ \text{for all $i, j$},
\end{align*}
where $\delta_{ij}$ denotes the Kronecker delta. Let $(M, \gM)$ be a Riemannian manifold. We denote by $\Delta_{g_M}$ the geometric Laplacian $d^\ast d$. In this paper, we introduce the following equation on $M$:
\begin{align}\label{generalized KW equation}
\Delta_{g_M}\xi+\sum_{j=1}^d a_je^{(\iota^\ast u^j, \xi)}\iota^\ast u^j=w,
\end{align}
where $\xi$ is a $k^\ast$-valued function on $M$ which is the solution of (\ref{generalized KW equation}) for given nonnegative functions $a_1, \dots, a_d$ and a $k^\ast$-valued  function $w$.  We give some examples of equation (\ref{generalized KW equation}).
\begin{ex}
Let $d=1$, $K=\U(1)$. Then equation (\ref{generalized KW equation}) is the Kazdan-Warner equation \cite{KW1}:
\begin{align*}
\Delta_{g_M} f+he^f=c.
\end{align*}
It should be noted that in \cite{KW1} the sign of a given function $h$ is not assumed to be nonnegative.
\end{ex}
\begin{ex}
Let $K$ be a connected subtorus of $T^d$ which is defined as $K\coloneqq \{(g_1, \dots, g_d)\in T^d\mid g_1\cdots g_d=1\}$. We consider equation (\ref{generalized KW equation}) on an open subset $U$ of the complex plane $\C\simeq\R^2$ with the standard metric $g_{\R^2}\coloneqq dx\otimes dx+dy\otimes dy$. The Laplacian $\Delta_{g_{\R^2}}$ of the standard metric $g_{\R^2}$ is given as follows:
\begin{align*}
\Delta_{g_{\R^2}}=-\left(\frac{\partial^2}{\partial^2 x}+\frac{\partial^2}{\partial^2 y}\right)=-4\frac{\partial^2}{\partial z. \partial \bar{z}}.
\end{align*}
We set $a_1=\cdots =a_d=4, \ w=0$. Then equation (\ref{generalized KW equation}) is the following:
\begin{align}
-\frac{\partial^2}{\partial z \partial \bar{z}}\xi+\sum_{j=1}^de^{(\iota^\ast u^j, \xi)}\iota^\ast u^j=0. \label{Toda-type}
\end{align}
We show that equation (\ref{Toda-type}) is equivalent to the two-dimensional periodic Toda lattice with opposite sign \cite{GL1}. We first define a surjection $\pi: T^d \longrightarrow K $ by
\begin{align*}
\pi(g_1,\dots, g_d)=(g_1^{-1}g_2, g_2^{-1}g_3, \dots, g_d^{-1}g_1).
\end{align*}
The derivative $\pi_\ast: t^d\rightarrow k$ of the map $\pi$ is 
\begin{align*}
\pi_\ast(u_j)=-u_j+u_{j-1} \ \text{for $j=1,\dots, d$},
\end{align*}
where we denote by $u_0$ the vector $u_d$. Let $(\pi_\ast)^\ast:k^\ast\rightarrow (t^d)^\ast$ be the adjoint of the derivative $\pi_\ast$. Then one can check that the following holds for each $j$:
\begin{align*}
(\pi_\ast)^\ast(\iota^\ast u^j)=u^{j+1}-u^j.
\end{align*}
By using the adjoint $(\pi_\ast)^\ast$, we identify $k^\ast$ with a subset $\{\theta_1u^1+\cdots+\theta_d u^d\in (t^d)^\ast\mid\theta_1+\cdots +\theta_d=0\}$ of $(t^d)^\ast$. A $k^\ast$-valued function $\xi:U\rightarrow k^\ast$ is then identified with real valued functions $\xi_1,\dots, \xi_d$ satisfying $\xi_1+\cdots+\xi_d=0$. Further under this identification, equation (\ref{Toda-type}) is equivalent to the following:
\begin{align}
\sum_{j=1}^d \left\{\frac{\partial^2}{\partial z \partial \bar{z}}\xi_j+e^{\xi_{j+1}-\xi_j}-e^{\xi_j-\xi_{j-1}}\right\}u^j=0. \label{toda equation}
\end{align}
Equation (\ref{toda equation}) is known as the two-dimensional periodic Toda lattice with opposite sign \cite{GL1}.
\end{ex}
Therefore equation (\ref{generalized KW equation}) can be considered as a generalization of the above examples. We call equation (\ref{generalized KW equation}) {\it generalized Kazdan-Warner equation.} We solve equation (\ref{generalized KW equation}) on any compact Riemannian manifold under the following assumption on $a_1, \dots, a_d$: 
\begin{enumerate}[($\ast$)]
\item \label{coefficients}  For each $j\in J_a$, $a_j^{-1}(0)$ is a set of measure $0$ and $\log a_j$ is integrable,
\end{enumerate} 
where $J_a$ denotes $\{j\in\{1, \dots, d\}\mid \text{$a_j$ is not identically 0}\}$.
Note that if $M$ is a complex manifold with a holomorphic hermitian bundle $(E, h) \rightarrow M$, then $a_1=|\Phi_1|^2, \dots, a_d=|\Phi_d|^2$ satisfy condition $(\ast)$ for any holomorphic sections $\Phi_1, \dots, \Phi_d$ of $E$. Our main theorem is the following:
\begin{theo}\label{main theorem}
{\it 
Let $(M, g_M)$ be an $m$-dimensional compact connected Riemannian manifold. We take non-negative $C^\infty$ functions $a_1, \dots, a_d $ and  a $k^\ast$-valued $C^\infty$ function $w$. Assume $a_1, \dots, a_d$ satisfy condition $(\ast)$. Then the following (1) and (2) are equivalent:
\begin{enumerate}[(1)]
\item \label{GKW equation has a solution}The generalized Kazdan-Warner equation has a $C^\infty$ solution $\xi:M\rightarrow k^\ast$:
\begin{align}\label{KWequation}
\Delta_{g_M}\xi+\sum_{j=1}^d a_je^{(\iota^\ast u^j, \xi)}\iota^\ast u^j=w;
\end{align}
\item \label{wbar is in the interior of the cone}
The given functions $a_1, \dots, a_d$ and $w$ satisfy
\begin{align}
\int_M w \ \vol\in\sum_{j\in J_a}\R_{>0}\iota^\ast u^j, \label{an important condition}
\end{align}
where $\mu_{g_M}$ denotes the measure induced by $g_M$.
\end{enumerate}
Moreover if $\xi$ and $\xi^\prime$ are $C^\infty$ solutions of equation (\ref{KWequation}), then $\xi-\xi^\prime$ is a constant which is in the orthogonal complement of $\sum_{j\in J_a}\R\iota^\ast u^j$.}
\end{theo}
\begin{rem}
We comment on the subtorus $K$. The closed connected subtorus $K$ is isomorphic to a real torus $T^n$. Therefore one can consider that our equation (\ref{generalized KW equation}) is associated with an embedding $i: T^n\rightarrow T^d$. More generally, for a homomorphism $\tau: T^{n^\prime}\rightarrow T^d$, we can define a differential equation in the same way as the definition of equation (\ref{generalized KW equation}). However, it should be noted that for given two homomorphisms: $\tau_1:T^{n_1}\rightarrow T^d$ and $\tau_2:T^{n_2}\rightarrow T^d$, if their image coincides: $\Im \tau_1=\Im \tau_2$, then equations which are associated with $\tau_1$ and $\tau_2$ are equivalent. Therefore essentially, it is enough to consider equations associated with a closed connected subtorus $K$ of $T^d$. 
\end{rem}
\begin{rem}
We mention the definition of  our equation (\ref{generalized KW equation}). Let $\Exp:t^d\rightarrow T^d$ be the exponential map which is defined as
\begin{align*}
\Exp(v)=(e^{\inum\langle v, u^1\rangle}, \dots, e^{\inum\langle v, u^d\rangle}) \ \text{for $v\in t^d$}, 
\end{align*}
where we denote by $\langle \cdot, \cdot\rangle$ the natural coupling of $t^d$ and $(t^d)^\ast$. We denote by $t^d_\Z$ a lattice $\ker\Exp$ of $t^d$. Then the subtorus $K$ defines a sublattice $k_\Z\coloneqq \ker\Exp\left.\right|_k$ of the lattice $t^d_\Z$. Conversely, for a given sublattice $k_\Z$ of $t_\Z^d$ we can define a closed connected subtorus $K$ of $T^d$ as $K\coloneqq \Exp(k_\Z\otimes\R)$. Given two sublattices $k_\Z$ and $k^\prime_\Z$ define the same subtorus if and only if they give the same rational subspace: $k_\Z\otimes\Q=k^\prime_\Z\otimes\Q$. Therefore we see that to give a closed connected subtorus $K$ of $T^d$ is equivalent to give a rational subspace of $t_\Z^d\otimes\Q$. Hence our equation (\ref{generalized KW equation}) can be considered as an equation which is associated with a rational subspace of $t_\Z^d\otimes \Q$. We note that our generalized Kazdan-Warner equation (\ref{generalized KW equation}) can be defined not only for a rational subspace of $t^d_\Z\otimes\Q$ but also for any real vector subspace of $t^d$. More specifically, for a given real vector subspace $V\subseteq t^d$, by using the dual of the inclusion $\iota: V\rightarrow t^d$ we can define a differential equation in the same way as the definition of equation (\ref{generalized KW equation}). Further it should be remarked that our Theorem \ref{main theorem} also holds for equation (\ref{generalized KW equation}) which is associated with any real vector subspace $V$ of $t^d$. Finally, we remark that to give an embedding of real vector space $\R^n\rightarrow t^d$ is equivalent to give $d$-generators $v_1,\dots, v_d$ of $(\R^n)^\ast$. This is because the generators $v_1,\dots, v_d$ defines a linear surjection 
\begin{align*}
p:(t^d)^\ast&\longrightarrow (\R^n)^\ast \\
r_1u^1+\cdots +r_du^d&\longmapsto r_1v_1+\cdots+r_dv_d
\end{align*}
and its dual $p^\ast:\R^n\rightarrow t^d$ defines an embedding. Therefore one can consider that our equation (\ref{generalized KW equation}) is defined for a real vector space $(\R^n)^\ast$ and its $d$-generators $v_1,\dots, v_d$.
\end{rem}
We note that statement (\ref{GKW equation has a solution}) of Theorem \ref{main theorem} immediately implies statement (\ref{wbar is in the interior of the cone}): If (\ref{GKW equation has a solution}) holds, by integrating both sides of equation (\ref{KWequation}), we have 
\begin{align*}
\sum_{j=1}^d \left(\int_M a_je^{(\iota^\ast u^j, \xi)}\ \vol\right)\iota^\ast u^j=\int_M w \ \vol.
\end{align*}
Hence it suffices to solve equation (\ref{KWequation}) under the assumption of (\ref{wbar is in the interior of the cone}) and to prove the uniqueness of the solution up to a constant which is in the orthogonal complement of a vector subspace $\sum_{j\in J_a}\R\iota^\ast u^j$. We give a proof of Theorem \ref{main theorem} in Section \ref{proof section}.

We mention the relationship between our generalized Kazdan-Warner equation and the abelian GLSM. It is well known that the Kazdan-Warner equation is related to the vortex equation which is associated with the standard action of $\U(1)$ on $\C$ (see \cite{Br1}, \cite{JT1}, and \cite{Tau1}). The present paper generalizes this relationship. For the diagonal action of a torus $K$ on $\C^d$ we have a generalization of the vortex equation which is investigated in \cite{Bap1}. In \cite{Bap1}, such a generalization of the vortex equation is called the abelian GLSM. In \cite[Theorem 3.2]{Bap1}, the Hitchin-Kobayashi correspondence for the abelian GLSM is obtained by using the general theory of the Hitchin-Kobayashi correspondence which is given in \cite{Ban1} and \cite{Mun1}. Our Theorem \ref{main theorem} gives a direct proof of \cite[Theorem 3.2]{Bap1}.

We also note that our Theorem \ref{main theorem} gives another proof of a theorem of Baraglia \cite{Bar1} which asserts that a cyclic Higgs bundle gives a solution of the periodic Toda equation. This will be explained in Section 2.

At the end of the introduction, we remark that a special case of our generalized Kazdan-Warner equations appears in \cite{BW1} and \cite{Doa1} in relation to a generalization of the Seiberg-Witten equation. In \cite{BW1}, for the case that the torus $K$ is given as 
\begin{align*}
K=\{(g, \dots, g, g^{-1},\dots,  g^{-1})\in T^{2d}\mid g\in \U(1)\},
\end{align*} 
our generalized Kazdan-Warner equation is solved on any compact Riemannian manifold  under a different assumption for given functions.

\medskip
\noindent
{\bf Acknowledgement.} The author is grateful to his supervisor Professor Ryushi Goto for fruitful discussions and encouragements. He also wishes to express his gratitude to anonymous referee for careful reading and helpful suggestions. 
\section{Cyclic Higgs bundles}
\label{sec:1}
In this  section, we see that our Theorem \ref{main theorem} gives another proof of a theorem of Baraglia \cite{Bar1} which asserts that a cyclic Higgs bundle gives a solution of the periodic Toda equation. 

Let $X$ be a compact connected Riemann surface and $K_X\rightarrow X$ the canonical line bundle. Assume that $X$ has genus $g(X)\geq 2$. Let $\g$ be a complex simple Lie algebra of rank $l$ and $G=\Aut(\g)_0$ the identity component of the automorphism group of $\g$. We define a $G$-Higgs bundle over X as a pair of a holomorphic principal $G$-bundle $P_G\rightarrow X$ and a holomorphic section $\Phi$ of $\ad(P_G)\otimes K_X\rightarrow X$. Here we denote by $\ad(P_G)\rightarrow X$ the vector bundle associated with the adjoint representation of $G$ on $\g$. The holomorphic section $\Phi$ is called a Higgs field. Let $B(\cdot, \cdot)$ be the Killing form on $\g$ and $\rho:\g\rightarrow \g$ an antilinear involution such that $-B(\cdot, \rho(\cdot))$ defines a positive-definite hermitian metric on $\g$. We define a maximal compact subgroup $G_\rho$ as $G_\rho\coloneqq \{g\in G\mid gg^\ast=1\}$, where we denote by $g^\ast$ the adjoint of $g$ with respect to the hermitian metric $-B(\cdot, \rho(\cdot))$. The following is known as the Hitchin-Kobayashi correspondence for Higgs bundles:
\begin{theo}{\rm (\label{Kobayashi-Hitchin}\cite{Hit1}, \cite{Sim1})}
{\it For a $G$-Higgs bundle $(P_G, \Phi)$, the following (i) and (ii) are equivalent.
\begin{enumerate}[(i)]
\item The $G$-Higgs bundle $(P_G, \Phi)$ is polystable.
\item There exists a $G_\rho$-subbundle $P_{G_\rho}\subseteq P_G$ such that a connection $\nabla$ which is defined as
\begin{align*}
\nabla\coloneqq A+\Phi-\rho(\Phi)
\end{align*}
is a flat connection, where we denote by $A$ the canonical connection of $P_{G_\rho}$.
\end{enumerate}
}
\end{theo}
We refer the reader to \cite{GO1}, \cite{Hit1} and \cite{Sim1}  for the definition of stability of Higgs bundles. We note that the connection $\nabla$ defined in Theorem \ref{Kobayashi-Hitchin} is flat if and only if the connection $A$ satisfies the following Hitchin's self-duality equation.
\begin{align*}
F_A+[\Phi\wedge(-\rho)(\Phi)]=0,
\end{align*} 
where we denote by $F_A$ the curvature of $A$. 

We give an example of a $G$-Higgs bundle introduced by Hitchin \cite{Hit2}. Let $\h\subseteq\g$ be a Cartan subalgebra of $\g$ and $\Delta\subseteq \h^\ast$ the root system. We denote by $\g=\h\oplus\bigoplus_{\alpha\in\Delta}\g_\alpha$ the root space decomposition. We fix a base of $\Delta$ which is denoted as $\Pi=\{\alpha_1,\dots, \alpha_l\}$. Let $\epsilon_1, \dots, \epsilon_l$ be the dual basis of $\alpha_1,\dots, \alpha_l$. We define a semisimple element $x$ of $\g$ as follows:
\begin{align*}
x\coloneqq\sum_{i=1}^l\epsilon_i.
\end{align*}
For each $\alpha\in\Delta$, let $h_\alpha\in\h$ be the coroot which is defined as $\beta(h_\alpha)=2B(\beta, \alpha)/B(\alpha, \alpha)$ for $\beta\in\Delta$. We take a basis $(e_\alpha)_{\alpha\in\Delta}$ of $\bigoplus_{\alpha\in\Delta}\g_\alpha$ which satisfies the following:
\begin{align*}
[e_\alpha, e_{-\alpha}]=h_\alpha \ \text{for each $\alpha\in\Delta$}.
\end{align*}
We define an antilinear involution $\rho: \g\rightarrow \g$ as follows:
\begin{align*}
\rho(h_\alpha)=-h_\alpha, \ \rho(e_\alpha)=-e_{-\alpha} \ \text{for each $\alpha\in\Delta$}.
\end{align*}
We can check that $-B(\cdot, \rho(\cdot))$ is a positive-definite hermitian metric on $\g$. The semisimple element $x$ is denoted as 
\begin{align*}
x=\sum_{i=1}^l r_i h_{\alpha_i}
\end{align*}
for some positive $r_1, \dots, r_l$. We define nilpotent elements $e$ and $\tilde{e}$ of $\g$ as follows:
\begin{align*}
&e\coloneqq\sum_{i=1}^l\sqrt{r_i}e_{\alpha_i}, \\
&\tilde{e}\coloneqq\sum_{i=1}^l\sqrt{r_i}e_{-\alpha_i}.
\end{align*}
Then we can check that we have the following commutation relations for $x, e, \tilde{e}$:
\begin{align*}
[x, e]=e, \ [x, \tilde{e}]=-\tilde{e}, \ [e,\tilde{e}]=x.
\end{align*}
Let $\s$ be the subalgebra which is spanned by $x, e, \tilde{e}$. Then $\s$ is one of the principal three dimensional subalgebras introduced by Kostant \cite{Kos1}. We can check that the adjoint representation of $\s$ on $\g$ has $l$-irreducible subspaces denoted as
\begin{align*}
\g=\bigoplus_{i=1}^l V_i.
\end{align*}
Moreover, we can also check that the dimension of $V_i$ is odd for each $i\in\{1,\dots, l\}$. We denote by $m_i$ the integer satisfying $\dim V_i=2m_i+1$ and we assume that we have $m_1\leq\cdots\leq m_l$. We note that $m_l$ is nothing but the height of the highest root $\delta$. We define an embedding of $\C^\ast$ into $G$ as follows:
\begin{align*}
\tau: \C^\ast&\longrightarrow G \\
e^z&\longmapsto \Exp(zx),
\end{align*}
where we denote by $\Exp:\g\rightarrow G$ the exponential map defined as $\Exp(v)\coloneqq \exp({[v, \cdot]})$ for $v\in\g$. Let $P_{K_X}$ be the holomorphic frame bundle of the canonical line bundle $K_X$. We define a holomorphic principal $G$-bundle $P_{G, K_X}$ as 
\begin{align*}
P_{G, K_X}\coloneqq P_{K_X}\times_\tau G.
\end{align*}
Then the adjoint bundle $\ad(P_{G, K_X})$ decomposes as follows:
\begin{align*}
\ad(P_{G, K_X})=\bigoplus_{m=-m_l}^{m_l}\g_m\otimes K_X^m,
\end{align*}
where we denote by $\g=\bigoplus_{m=-m_l}^{m_l}\g_m$ the eigenspace decomposition of the adjoint action of $x$. Let $e_1, \dots, e_l$ be the highest weight vectors of $V_1, \dots, V_l$. We note that for each $i$, the vector $e_i$ lies in $\g_{m_i}$. We let $e_l=e_\delta$. For a given $q=(q_1\dots, q_l)\in \bigoplus_{i=1}^lH^0(K_X^{m_i+1})$, we define a Higgs field $\Phi(q)\in H^0(\ad(P_{G, K_X})\otimes K_X)$ as follows:
\begin{align*}
\Phi(q)\coloneqq \tilde{e}+q_1e_1+\cdots +q_l e_l. 
\end{align*}
Then we have a $G$-Higgs bundle $(P_{G, K_X}, \Phi(q))$ which is parametrized by $q=(q_1,\dots, q_l)\in \bigoplus_{i=1}^lH^0(K_X^{m_i+1})$. In \cite{Hit2}, Hitchin proved that the Higgs bundle $(P_{G, K_X}, \Phi(q))$ is stable for any $q\in \bigoplus_{i=1}^lH^0(K_X^{m_i+1})$. 

It is well known that there exists a set of homogeneous polynomials $p_1, \dots, p_l\in\Sym(\g^\ast)$ of $\deg(p_i)=m_i+1$ for $i=1, \dots, l$ which satisfies
\begin{align*}
\C[p_1, \dots, p_l]=A(\g)^G,
\end{align*}
where we denote by $A(\g)^G$ the invariant polynomial ring. Moreover the generators $p_1,\dots, p_l$ can be chosen so that 
\begin{align}
p_i(\tilde{e}+z_1e_1+\cdots +z_l e_l)=z_j \ \text{for any $z_1,\dots, z_l\in\C$.}\label{polynomials}
\end{align}
Then we have the following map which is called the {\it Hitchin fibration}:
\begin{align*}
\M_G&\longrightarrow \bigoplus_{i=1}^lH^0(K_X^{m_i+1}) \\
(P_G, \Phi)&\longmapsto (p_1(\Phi), \dots, p_l(\Phi)),
\end{align*}
where we denote by $\M_G$ the moduli space of $G$-Higgs bundles. Since we have (\ref{polynomials}), the following map defines a section of the Hitchin fibration, which is called the {\it Hitchin section}:
\begin{align*}
\bigoplus_{i=1}^lH^0(K_X^{m_i+1})&\longrightarrow \M_G\\
q=(q_1,\dots, q_l)&\longmapsto (P_{G, K_X}, \Phi(q)).
\end{align*}
We refer the reader to \cite{Hit2} for the details of the Hitchin section.

For each $q=(q_1, \dots, q_l)\in\bigoplus_{i=1}^lH^0(K_X^{m_i+1})$, let $(P_{G, K_X}, \Phi(q))$ be the Higgs bundle constructed as above. In \cite{Bar1}, the Higgs bundle $(P_{G, K_X}, \Phi(q))$ is called a {\it cyclic Higgs bundle} if the parameter $q=(q_1,\dots, q_l)$ satisfies the following:
\begin{align*}
q_1=\cdots=q_{l-1}=0.
\end{align*}
We take a $\U(1)$-subbundle $P_{\U(1)}\subseteq P_{K_X}$. Let $P^0_{G_\rho}$ be the $G_\rho$-subbundle $P_{\U(1)}\times_\tau G_\rho$ of $P_{G, K_X}$. In \cite{Bar1}, Baraglia proved the following by using properties of the Kostant's principal element and the uniqueness of the solution of the Hitchin's self-duality equation:

\begin{theo}\label{Baraglia's Theorem}{\rm (\cite{Bar1})} {\it Let $\sigma\in \Aut(P_{G, K_X})$ be a gauge transformation such that the $G_\rho$-subbundle $\sigma^{-1}(P^0_{G_\rho})$ of $P_{G, K_X}$ satisfies condition (ii) of Theorem \ref{Kobayashi-Hitchin}. Then we have 
\begin{align*}
\sigma^\ast\sigma\in C^\infty(X, H),
\end{align*}
where we denote by $H$ the maximal torus $\Exp(\h)$.
}
\end{theo}
We show that Theorem \ref{Baraglia's Theorem} follows from our Theorem \ref{main theorem}. Let $\h_\R$ be the real subspace of the Cartan subalgebra $\h$ generated by $(h_\alpha)_{\alpha\in\Delta}$. For a smooth map $\Omega:X\rightarrow \h_\R$, we denote by $\sigma_\Omega$ the gauge transformation $\Exp(\Omega)$. Theorem \ref{Baraglia's Theorem} says that there exists an $\Omega\in C^\infty (X, \h_\R)$ such that a $G_\rho$-subbundle $\sigma_\Omega^{-1}(P^0_{G_\rho})$ of $P_{G, K_X}$ satisfies condition (ii) of Theorem \ref{Kobayashi-Hitchin}. We denote by $A$ the canonical connection of $P^0_{G_\rho}$. The Hitchin's self-duality equation for the subbundle $\sigma_\Omega^{-1}(P^0_{G_\rho})$ is the following:
\begin{align}
F_A+2\partial\bar{\partial}\Omega-\sum_{\alpha\in\Pi\cup\{-\delta\}}(\Phi_\alpha\wedge\bar{\Phi}_\alpha) e^{2\alpha(\Omega)}h_\alpha=0, \label{Hitchin-Toda equation}
\end{align}
where the Higgs field $\Phi(q)=\tilde{e}+q_le_l$ is denoted as $\Phi(q)=\sum_{\alpha\in\Pi\cup\{-\delta\}}\Phi_\alpha e_\alpha$. We refer the reader to \cite{AF1}, \cite{Bar1} and \cite{GH1} for the relation between equation (\ref{Hitchin-Toda equation}) and the periodic Toda lattice. Let $\omega_X$ be a K\"ahler form and $\Lambda_{\omega_X}$ the adjoint $(\omega_X\wedge)^\ast$. Then (\ref{Hitchin-Toda equation}) is equivalent to the following:
\begin{align}
\Delta_{\omega_X} \Omega+\sum_{i=1}^lr_ie^{2\alpha_i(\Omega)}h_{\alpha_i}+|q_l|^2_{\omega_X}e^{-2\delta(\Omega)}h_{-\delta}=\inum \Lambda_{\omega_X} F_A .\label{Hitchin-Toda equation 2}
\end{align}
We note that equation (\ref{Hitchin-Toda equation 2}) is a special case of our generalized Kazdan-Warner equations which is associated with the action of a real torus $\Exp(\inum \h_\R)$ on $\bigoplus_{\alpha\in \Pi\cup\{-\delta\}}\g_\alpha$. Then applying Theorem \ref{main theorem}, we see that equation (\ref{Hitchin-Toda equation 2}) has a solution if $q_l\neq 0$ since we have 
\begin{align*}
\sum_{\alpha\in\Pi\cup\{-\delta\}}\R_{>0}h_\alpha=\h_\R.
\end{align*}
We see that equation (\ref{Hitchin-Toda equation 2}) has a solution also in the case that $q_l=0$ since we have
\begin{align*}
\frac{\inum}{2\pi}\int_X F_A=(2g(X)-2) x
\end{align*}
and the semisimlple element $x$ lies in $\sum_{i=1}^l\R_{>0}h_{\alpha_i}$. We note that if $q_l=0$ from \cite{KW1} we see that equation (\ref{Hitchin-Toda equation}) has a solution since at this time equation (\ref{Hitchin-Toda equation 2}) is equivalent to $l$-copies of the Kazdan-Warner equation. 

At the end of this section, we refer the reader to \cite{DL1} for a generalization of the Baraglia's cyclic Higgs bundles,  and we note that the cyclic Higgs bundle which is constructed as above appears in \cite{Sim2} as a special case of the cyclotomic Higgs bundles introduced in \cite{Sim2}. 

\section{Proof of Theorem \ref{main theorem}}\label{proof section}
\subsection{Outline of the proof}
We outline the proof of Theorem \ref{main theorem}. As we have already mentioned, it suffices to solve equation (\ref{KWequation})  under the assumption of (\ref{an important condition})
and to prove the uniqueness of the solution up to a constant which is in the orthogonal complement of $\sum_{j\in J_a}\R\iota^\ast u^j$. We prove this using the variational method. We first define a functional $E$ whose critical point is a solution of equation (\ref{KWequation}). Then we show that the functional $E$ is convex. The uniqueness of the solution of the equation follows from the convexity of $E$. Secondly, in Lemma \ref{Key Lemma} we show that the functional $E$ is bounded below and moreover the following estimate holds:
\begin{align*}
|\xi|_{L^2}\leq(E(\xi)+C)^2+C^\prime E(\xi)+C^{\prime\prime}
\end{align*}
with some constants $C$, $C^\prime$ and $C^{\prime\prime}$. We use (\ref{an important condition}) in the proof of Lemma \ref{Key Lemma}. Finally by using a method developed in \cite{Br2} we see that the functional $E$ has a critical point. 
\subsection{Proof of Theorem \ref{main theorem}}
\noindent
Hereafter we normalize the measure $\mu_{g_M}$ so that the total volume is 1. 
\begin{align*}
\Vol(M, g_M)\coloneqq\int_M1 \ d\mu_{g_M}=1.
\end{align*} 
\begin{defi}
We define a functional $E: L^{2m}_3(M, k^\ast)\rightarrow \R$ by
\begin{align*}
E(\xi)\coloneqq\int_M \Bigl\{\frac{1}{2}(d\xi, d\xi)+\sum_{j=1}^da_je^{(\iota^\ast u^j, \xi)}-(w, \xi)\Bigr\}\ \vol \ \text{for $\xi\in L^{2m}_3(M, k^\ast)$.}
\end{align*}

\end{defi}

\begin{lemm}\label{critical point}
{\it
For each $\xi\in L^{2m}_3(M, k^\ast)$, the following are equivalent:
\begin{enumerate}[(1)]
\item \label{critical point 1}$\xi$ is a critical point of $E$;
\item \label{critical point 2}$\xi$ solves equation (\ref{KWequation}).
\end{enumerate}
Moreover if $\xi$ solves equation (\ref{KWequation}), then $\xi$ is a $C^\infty$ function.
}
\end{lemm}

\begin{proof}
We have the following for each $\eta\in L^{2m}_3(M, k^\ast)$:
\begin{align*}
\left.\frac{d}{dt}\right|_{t=0}E(\xi+t\eta)=\int_M (\Delta_{g_M}\xi+\sum_{j=1}^d a_j e^{(\iota^\ast u^j, \xi)}\iota^\ast u^j-w, \eta) \ \vol.
\end{align*}
Therefore (\ref{critical point 1}) and (\ref{critical point 2}) are equivalent. The rest of the proof follows from the elliptic regularity theorem.
\end{proof}

\begin{lemm}\label{convexity}
{\it
For each $\xi, \eta\in L^{2m}_3(M, k^\ast)$ and $t\in\R$, the following holds:
\begin{align*}
\frac{d^2}{dt^2}E(\xi+t\eta)\geq0.
\end{align*}
Moreover the following are equivalent:
\begin{enumerate}[(1)]
\item There exists a $t_0\in\R$ such that $\left. \frac{d^2}{dt^2}\right|_{t=t_0}E(\xi+t\eta)= 0$;
\item $\frac{d^2}{dt^2}E(\xi+t\eta)= 0 $ for all $t\in \R$;
\item $\eta$ is a constant which is in the orthogonal complement of $\sum_{j\in J_a}\R\iota^\ast u^j$.
\end{enumerate}
}
\end{lemm}
\begin{proof}
A direct computation shows that
\begin{align*}
\frac{d^2}{dt^2}E(\xi+t\eta)=\int_M\{(d\eta, d\eta)+\sum_{j=1}^da_je^{(\iota^\ast u^j, \xi+t\eta)}(\iota^\ast u^j, \eta)^2 \}\ \vol.
\end{align*}
This implies the claim.
\end{proof}
\begin{cor}
{\it
Let $\xi$ and $\xi^\prime$ be $C^\infty$ solutions of equation (\ref{KWequation}). Then $\xi-\xi^\prime$ is a constant which is in the orthogonal complement of $\sum_{j\in J_a}\R\iota^\ast u^j$. 
}
\end{cor}
\begin{proof}
From Lemma \ref{critical point}, we have the following:
\begin{align*}
\left. \frac{d}{dt}\right|_{t=0}E(t\xi+(1-t)\xi^\prime)=\left. \frac{d}{dt}\right|_{t=1}E(t\xi+(1-t)\xi^\prime)=0.
\end{align*}
Then Lemma \ref{convexity} gives the result.
\end{proof}
Let $W\subseteq k^\ast$ be the vector subspace of $k^\ast$ which is generated by $(\iota^\ast u^j)_{j\in J_a}$. Hereafter we assume that $W=k^\ast$ for simplicity. We can make the assumption since if the vector subspace $W$ is strictly smaller than the vector space $k^\ast$, then by restricting the domain of the functional $E$ to the subspace $L^{2m}_3(M,W)$ of $L^{2m}_3(M, k^\ast)$ we have the same proof as in the case that $W=k^\ast$. 
\begin{lemm}\label{Key Lemma}
{\it
The functional $E$ is bounded below. Further there exist non-negative constants $C$, $C^{\prime}$ and $C^{\prime \prime}$ such that
\begin{align*}
|\xi|_{L^2}\leq(E(\xi)+C)^2+C^\prime E(\xi)+C^{\prime\prime}
\end{align*}
for all $\xi\in L^{2m}_3(M, k^\ast)$.
}
\end{lemm}
\begin{proof}
We first introduce a notation. For a $k^\ast$-valued function $\xi^\prime:M\rightarrow k^\ast$, we denote by $\bar{\xi^\prime}$ the average of $\xi^\prime$:
\begin{align*}
\bar{\xi^\prime}\coloneqq \int_M\xi^\prime\ \vol.
\end{align*}
Then the  following holds for each $\xi\in L^{2m}_3(M, k^\ast)$:
\begin{align*}
&E(\xi) \\
=&\int_M \Bigl\{\frac{1}{2}(d\xi, d\xi)+\sum_{j=1}^da_je^{(\iota^\ast u^j, \xi)}-(w, \xi)\Bigr\} \ \vol \\
=&\int_M \Bigl\{\frac{1}{2}(d\xi, d\xi)-(w, \xi-\bar{\xi})\Bigr\} \ \vol +\sum_{j=1}^d\int_Ma_je^{(\iota^\ast u^j, \xi)}\ \vol-(\bar{w},\bar{\xi}) \\
=&\int_M \Bigl\{\frac{1}{2}(d\xi, d\xi)-(w, \xi-\bar{\xi}) \Bigr\}\ \vol+\sum_{j\in J_a}\int_Me^{\log a_j}e^{(\iota^\ast u^j, \xi)}\ \vol-(\bar{w},\bar{\xi})  \\
\geq& \int_M \Bigl\{\frac{1}{2}(d\xi, d\xi)-(w, \xi-\bar{\xi}) \Bigr\}\ \vol+\sum_{j\in J_a}(e^{\int_M\log a_j\ \vol})e^{(\iota^\ast u^j, \bar{\xi})}-(\bar{w},\bar{\xi}), 
\end{align*}
where the final inequality follows from the Jensen's inequality. Since we have (\ref{an important condition}), there exist positive numbers $(s_j)_{j\in J_a}$ such that $\bar{w}=\sum_{j\in J_a}s_j \iota^\ast u^j$. Then we have the following:
\begin{align*}
&E(\xi) \\
\geq& \int_M \Bigl\{\frac{1}{2}(d\xi, d\xi)-(w, \xi-\bar{\xi}) \Bigr\}\ \vol+\sum_{j\in J_a}\left \{s_j^\prime e^{(\iota^\ast u^j, \bar{\xi})}-s_j(\iota^\ast u^j,\bar{\xi})\right\},
\end{align*}
where we denote by $s_j^\prime$ the coefficient $e^{\int_M\log a_j\ \vol}$ for each $j\in J_a$. We set $E_0$ and $E_1$ as follows: 
\begin{align*}
E_0(\xi)&\coloneqq \int_M \Bigl\{\frac{1}{2}(d\xi, d\xi)-(w, \xi-\bar{\xi}) \Bigr\}\ \vol \\
E_1(\xi)&\coloneqq \sum_{j\in J_a}\left\{s_j^\prime e^{(\iota^\ast u^j, \bar{\xi})}-s_j(\iota^\ast u^j,\bar{\xi})\right \}.
\end{align*}
We note that $E_1$ depends only on the average of $\xi$. Then the Poincar\'{e} inequality implies that $E_0$ is bounded below. We see that $E_1$ is also bounded below since the following $f$ is bounded below for any $s, s^\prime\in\R_{>0}$:
\begin{align*}
f(y)\coloneqq s^\prime e^y-sy \ \ \text{for $y\in\R$}.
\end{align*}
Therefore $E$ is bounded below. Further we have the following: for each $\bar{\xi}\neq0$,  we see $\lim_{t\to\infty}E_1(t\bar{\xi})-|t\bar{\xi}|^{1/2}=\infty$. This implies that $E_1(\xi)-|\bar{\xi}|^{1/2}$ attains a minimum since we have the following:
\begin{align*}
\min_{\xi\in L^{2m}_3(M, k^\ast)}\{E_1(\xi)-|\bar{\xi}|^{1/2}\}=\min_{|\bar{\xi}|=1}\min_{t\in\R}\{E_1(t\bar{\xi})-|t\bar{\xi}|^{1/2}\}.
\end{align*}
In particular, $E_1(\xi)-|\bar{\xi}|^{1/2}$ is bounded below. This implies that there exists a constant $C$ such that 
\begin{align*}
|\bar{\xi}|\leq (E(\xi)+C)^2
\end{align*}
for all $\xi\in L^{2m}_3(M, k^\ast)$. We also obtain the following estimate for some $C^\prime$ and $C^{\prime \prime}$ from the Poincar\'{e} inequality:
\begin{align*}
|\xi-\bar{\xi}|_{L^2}\leq C^\prime E(\xi)+C^{\prime\prime}.
\end{align*}
Then we have
\begin{align*}
|\xi|_{L^2}&\leq |\bar{\xi}|+|\xi-\bar{\xi}|_{L^2}\leq (E(\xi)+C)^2+C^\prime E(\xi)+C^{\prime\prime}
\end{align*}
and this implies the claim.
\end{proof}
We note that the following method was originally developed by Bradlow \cite{Br2}:
\begin{defi}
Let $B>0$ a positive real number. We define a subset $L^{2m}_3(M, k^\ast)_B$ of $L^{2m}_3(M, k^\ast)$ by
\begin{align*}
L^{2m}_3(M, k^\ast)_B\coloneqq\{\xi\in L^{2m}_3(M, k^\ast)\mid |\Delta_{g_M}\xi+\sum_{j=1}^da_je^{(\iota^\ast u^j, \xi)}\iota^\ast u^j-w|_{L^{2m}_1}^{2m}\leq B\}.
\end{align*}
\end{defi}
Then we have the following Lemma \ref{local minimum}. For the proof of Lemma \ref{local minimum}, we refer the reader to \cite[Lemma 3.4.2]{Br2} and \cite[Proposition 3.1]{AG1}.
\begin{lemm}\label{local minimum}
{\it 
If $E|_{L^{2m}_3(M, k^\ast)_B}$ attains a minimum at $\xi_0\in L^{2m}_3(M, k^\ast)_B$, then $\xi_0$ is a critical point of $E$. 
}
\end{lemm}
\begin{proof}
We define a map $F: L^{2m}_3(M, k^\ast)\rightarrow L^{2m}_1(M, k^\ast)$ as $F(\xi)=\Delta_\gM\xi+\sum_{j=1}^da_je^{(\iota^\ast u^j, \xi)}\iota^\ast u^j-w$ for $\xi\in L^{2m}_3(M, k^\ast)$. Then its linearization at $\xi$ is the following: 
\begin{align*}
(DF)_\xi: L^{2m}_3(M, k^\ast)&\longrightarrow L^{2m}_1(M, k^\ast) \\
\eta&\longmapsto \Delta_\gM \eta+\sum_{j=1}^d a_je^{(\iota^\ast u^j, \xi)}(\iota^\ast u^j,\eta)\iota^\ast u^j.
\end{align*}
The linearization $(DF)_\xi$ satisfies the following for each $\eta, \eta^\prime\in L^{2m}_3(M, k^\ast)$:
\begin{align}
((DF)_\xi(\eta), \eta^\prime)_{L^2}=\int_M\Bigl\{(d\eta, d\eta^\prime)+\sum_{j=1}^da_je^{(\iota^\ast u^j, \xi)}(\iota^\ast u^j, \eta)(\iota^\ast u^j, \eta^\prime) \Bigr\}\ \vol \label{DF is self-adjoint},
\end{align}
where we denote by $(\cdot, \cdot)_{L^2}$ the $L^2$-inner product. (\ref{DF is self-adjoint}) says that the linearization $(DF)_\xi$ is a self-adjoint operator with respect to the $L^2$-inner product. We set $\eta=\eta^\prime$. Then we have
\begin{align}
((DF)_\xi(\eta), \eta)_{L^2}=\int_M\Bigl\{(d\eta, d\eta)+\sum_{j=1}^da_je^{(\iota^\ast u^j, \xi)}(\iota^\ast u^j, \eta)^2 \Bigr\}\ \vol. \label{linearization satisfies}
\end{align}
From (\ref{linearization satisfies}) we see that the linearization $(DF)_\xi$ is injective since if $\eta\in L^{2m}_3(M, k^\ast)$ satisfies $(DF)_\xi(\eta)=0$, then the right hand side of (\ref{linearization satisfies}) must be  0 and thus we have $\eta=0$. It should be noted that we have used here the assumption that $(\iota^\ast u^j)_{j\in J_a}$ generates $k^\ast$. Then we see that $(DF)_\xi$ is bijective since it is a formally self-adjoint elliptic operator.  Assume that $E|_{L^{2m}_3(M, k^\ast)_B}$ attains a minimum at $\xi_0\in L^{2m}_3(M, k^\ast)_B$. Since $(DF)_{\xi_0}$ is bijective, there uniquely exists an $\eta\in L^{2m}_3(M, k^\ast)$ such that 
\begin{align*}
(DF)_{\xi_0}(\eta)=-F(\xi_0).
\end{align*}
Assume that $\xi_0$ is not a critical point of $E$. Then we have $\eta\neq 0$. Let $\xi_t$ denotes a line $\xi_0+t\eta$ parametrized by $t\in\R$. We then have the following:
\begin{align*}
\left.\frac{d}{dt}\right|_{t=0}E(\xi_t)&=(F(\xi_0), \eta)_{L^2} \\
&=-((DF)_{\xi_0}(\eta), \eta) \\
&=- \int_M\Bigl\{(d\eta, d\eta)+\sum_{j=1}^da_je^{(\iota^\ast u^j, \xi)}(\iota^\ast u^j, \eta)^2 \Bigr\}\ \vol<0.
\end{align*}
Then for a sufficiently small $\epsilon>0$, the functional $E(\xi_t)$ strictly decreases with increasing $t\in(-\epsilon, \epsilon)$. We also have the following:
\begin{align*}
\left.\frac{d}{dt}\right|_{t=0}F(\xi_t)=(DF)_{\xi_0}(\eta)=-F(\xi_0).
\end{align*}
This implies that around 0, $|F(\xi_t)|_{L^{2m}_1}^{2m}$ decreases with increasing $t$ :
\begin{align*}
&\left.\frac{d}{dt}\right|_{t=0}|F(\xi_t)|_{L^{2m}_1}^{2m} \\
=& \left.\frac{d}{dt}\right|_{t=0}\int_M |dF(\xi_t)|^{2m}\ \vol 
+\left.\frac{d}{dt}\right|_{t=0}\int_M |F(\xi_t)|^{2m}\ \vol \\
=&-2m\int_M |dF(\xi_0)|^{2m}\ \vol 
-2m\int_M |F(\xi_0)|^{2m}\ \vol <0.\\
\end{align*}
Further this implies that for a sufficiently small $t>0$, $\xi_t$ satisfies the following:
\begin{align*}
&E(\xi_t)<E(\xi_0), \\
&|F(\xi_t)|^{2m}_{L^{2m}_1}\leq B.
\end{align*}
However, this contradicts the assumption that $E|_{L^{2m}_3(M, k^\ast)_B}$ attains a minimum at $\xi_0\in L^{2m}_3(M, k^\ast)_B$. Hence $\xi_0$ is a critical point of $E$. 
\end{proof}
Therefore the problem reduces to show that $E|_{L^{2m}_3(M, k^\ast)_B}$ attains a minimum. To see this, we prove the following Lemma \ref{final lemma}:

\begin{lemm}\label{final lemma}
{\it 
Let $(\xi_i)_{i\in\N}$ be a  sequence of $L^{2m}_3(M, k^\ast)_B$ such that 
\begin{align*}
\lim_{i\to\infty}E(\xi_i)=\inf_{\eta\in L^{2m}_3(M, k^\ast)_B}E(\eta).
\end{align*}
Then we have $\sup_{i\in\N}|\xi_i|_{L^{2m}_3}<\infty$.
}
\end{lemm}
Before the proof of Lemma \ref{final lemma}, we recall the following:
\begin{lemm}\label{known lemma}{\rm (\cite[pp.72-73]{LT1})} {\it 
Let $f\in C^2(M, \R)$ be a non-negative function. 
If
\begin{align*}
\Delta_{g_M} f \leq C_0 f+C_1
\end{align*}
holds for some $C_0\in \R_{\geq0}$ and $C_1 \in \R$, then there is a positive constant $C_2$, depending only on $g_M$ and $C_0$, such that
\begin{align*}
\max_{x\in M}f(x)\leq C_2(|f|_{L^1}+C_1).
\end{align*}
}
\end{lemm}
\begin{proof}[Proof of Lemma \ref{final lemma}] We first note that the functional space $L^{2m}_3(M, k^\ast)$ is contained in $C^2(M,k^\ast)$. We then have the following for each $i\in\N$:
\begin{align}
\frac{1}{2}\Delta_{g_M}|\xi_i|^2&=(\Delta_{g_M} \xi_i, \xi_i)-|d\xi_i|^2 \notag \\
&\leq (\Delta_{g_M}\xi_i+\sum_{j=1}^da_j\iota^\ast u^j-w, \xi_i)-(\sum_{j=1}^da_j\iota^\ast u^j-w, \xi_i) \notag \\
&\leq (\Delta_{g_M}\xi_i+\sum_{j=1}^da_je^{(\iota^\ast u^j, \xi_i)}\iota^\ast u^j-w, \xi_i)-(\sum_{j=1}^da_j\iota^\ast u^j-w, \xi_i), \label{final inequality}
\end{align}
where we have used the following inequality:
\begin{align*}
y\leq ye^y \ \text{ for any $y\in\R$.}
\end{align*}
From (\ref{final inequality}), we have 
\begin{align}
&\frac{1}{2}\Delta_{g_M}|\xi_i|^2 \notag \\
&\leq |\Delta_{g_M}\xi_i+\sum_{j=1}^da_je^{(\iota^\ast u^j, \xi_i)}\iota^\ast u^j-w||\xi_i|+|\sum_{j=1}^da_j\iota^\ast u^j-w||\xi_i| \notag \\
&\leq C_3|\xi_i|, \label{inequality}
\end{align}
for a constant $C_3$, since we have $L^{2m}_1(M, k^\ast)\subseteq C^0(M, k^\ast)$ and the following:
\begin{align*}
|\Delta_\gM \xi_i+\sum_{j=1}^d a_je^{(\iota^\ast u^j, \xi_i)}\iota^\ast u^j -w|_{L^{2m}_1}^{2m}\leq B. 
\end{align*}
From (\ref{inequality}) and an inequality $|\xi_i|\leq \frac{1}{2}(|\xi_i|^2+1)$ we have the following:
\begin{align*}
\Delta_\gM|\xi_i|^2\leq C_3|\xi_i|^2+C_3.
\end{align*}
Then from Lemma \ref{known lemma}, there exists a constant $C_4$ such that
\begin{align*}
\max_{x\in M}\{|\xi_i|^2(x)\}\leq C_4(||\xi_i|^2|_{L^1}+C_3)=C_4(|\xi_i|^2_{L^2}+C_3).
\end{align*}
Combining Lemma \ref{Key Lemma}, we see that there exists a constant $C_5$ which does not depend on $i$ such that
\begin{align}
\max_{x\in M}\{|\xi_i|^2(x)\}\leq C_5. \label{max}
\end{align}
From (\ref{max}) and the $L^p$-estimate we have 
\begin{align}
|\xi_i|_{L^{2m}_2}
&\leq C_6|\Delta_\gM \xi_i|_{L^{2m}}+C_7|\xi_i|_{L^1} \notag \\
&\leq  C_6|\Delta_\gM \xi_i+\sum_{j=1}^da_je^{(\iota^\ast u^j, \xi_i)}\iota^\ast u^j-w|_{L^{2m}} \notag \\
&+C_6|\sum_{j=1}^da_je^{(\iota^\ast u^j, \xi_i)}\iota^\ast u^j-w|_{L^{2m}} +C_8 \notag \\
&\leq  C_9 \label{inequality2}
\end{align}
for some constants $C_6, C_7, C_8$ and $C_9$. Then we have the result since we can repeat the same argument for $|\xi_i|_{L^{2m}_3}$ as in the proof of inequality (\ref{inequality2}).
\end{proof}

\begin{cor}
{\it
The functional $E$ has a critical point.
}
\end{cor}
\begin{proof}
We take a sequence $(\xi_i)_{i\in\N}$ so that
\begin{align*}
\lim_{i\to\infty}E(\xi_i)=\inf_{\eta\in L^{2m}_3(M, k^\ast)_B}E(\eta).
\end{align*}
Then by Lemma \ref{final lemma} we have $\sup_{i\in\N}|\xi_i|_{L^{2m}_3}<\infty$, and this implies that there exits a subsequence $(\xi_{i_j})_{j\in\N}$ such that $(\xi_{i_j})_{j\in\N}$ weakly converges a $\xi_\infty\in L^{2m}_3(M, k^\ast)_B$. Since the functional $E$ is continuous with respect to the weak topology, we have 
\begin{align}
E(\xi_\infty)=\inf_{\eta\in L^{2m}_3(M, k^\ast)_B}E(\eta). \label{xiinfty}
\end{align}
(\ref{xiinfty}) says that $E|_{L^{2m}_3(M, k^\ast)_B}$ attains a minimum at $\xi_\infty$ and thus from Lemma \ref{local minimum} we have the result.
\end{proof}
\appendix
\section{Geometric invariant theory and the moment maps for linear torus actions}
We give a brief review of the relationship between the geometric invariant theory and the moment maps for linear torus actions. In particular, we clarify the relationship between condition (\ref{an important condition}) and the stability condition of the geometric invariant theory. General references for this section are \cite{Dol1}, \cite{Kin1}, \cite{Kir1}, \cite{KN1}, \cite{MFK1}, \cite{Nak1} and \cite{New1}.
\subsection{Notation}
We first fix our notation. Let $K$ be a closed connected subtorus of a real torus $T^d:=\U(1)^d$ with the Lie algebra $k\subseteq t^d$. We denote by $\iota^\ast: (t^d)^\ast \rightarrow k^\ast$ the dual map of the inclusion map $\iota: k\rightarrow t^d$. Let $u_1,\dots, u_d$ be a basis of $t^d$ defined by
\begin{align*}
u_1\coloneqq&(\inum, 0, \dots, 0), \\
u_2\coloneqq&(0, \inum, 0,\dots, 0), \\
&\cdots \\
u_d\coloneqq&(0,\dots, 0,\inum).
\end{align*}
We denote by $u^1, \dots, u^d\in (t^d)^\ast$ the dual basis of $u_1, \dots, u_d$. Let $(\cdot, \cdot)$ be the metric on $t^d$ and $(t^d)^\ast$ satisfying
\begin{align*}
(u_i, u_j)=(u^i, u^j)=\delta_{ij} \ \ \text{for all $i, j$,}
\end{align*}
 where $\delta_{ij}$ denotes the Kronecker delta. The diagonal action of $T^d$ on $\C^d$ induces an action of $K$ which preserves the K\"ahler structure of $\C^d$. Let $\mu_K:\C^d\rightarrow k^\ast$ be a moment map for the action of $K$ which is defined by
\begin{align*}
\langle \mu_K(z), v \rangle=\frac{1}{2}g_{\R^{2d}}(\inum vz, z) \ \ \text{for $v\in k$},
\end{align*}
where we denote by $g_{\R^{2d}}(\cdot, \cdot)$ the standard metric of $\C^d\simeq \R^{2d}$, and by $\langle \cdot, \cdot \rangle$ the natural coupling. The moment map $\mu_K$ is also denoted as
\begin{align*}
\mu_K(z)=-\frac{1}{2}\sum_{j=1}^d\iota^\ast u^j |z_j|^2 \ \ \text{for $z=(z_1,\dots, z_d)\in\C^d$}.
\end{align*}
Let $T^d_\C\coloneqq (\C^\ast)^d$ be the complexification of $T^d$. We define the exponential map $\Exp: t^d\oplus \inum t^d\rightarrow T^d_\C$ by
\begin{align*}
\Exp(v+\inum v^\prime)=(e^{\inum\langle v+\inum v^\prime, u^1\rangle}, \dots, e^{\inum\langle v+\inum v^\prime, u^d\rangle}).
\end{align*}
We denote by $K_\C$ the complexification of $K$. Let $k_\Z\subseteq k$ be $\ker \Exp |_k$ and $(k_\Z)^\ast$ the dual. Note that $(k_\Z)^\ast$ is naturally identified with $\sum_{j=1}^d\Z \ (\iota^\ast u^j/2\pi)$. For each $\alpha\in (k_\Z)^\ast$, we define a character $\chi_\alpha : K_\C\rightarrow \C^\ast$ by
\begin{align*}
\chi_\alpha(\Exp(v+\inum v^\prime))=e^{2\pi\inum\langle v+\inum v^\prime, \alpha\rangle}.
\end{align*}

\subsection{Symplectic and GIT quotients }
Let $\alpha\in(k_\Z)^\ast$. We define an action of $K_\C$ on $\C^d\times \C$ by 
\begin{align*}
g\cdot (z, v)\coloneqq(gz, \chi_\alpha(g)^{-1}v) \ \ \text{for $(z, v)\in\C^d\times\C$.}
\end{align*}
Let $R_\alpha$ be the invariant ring for the above action:
\begin{align*}
R_\alpha\coloneqq\{\hat{f}(x, y)\in\C[x_1,\dots, x_d, y]\mid\text{$\hat{f}(g\cdot(x,y))=\hat{f}(x,y)$ for all $g\in K_\C$} \}.
\end{align*} 
By a theorem of Nagata, $R_\alpha$ is finitely generated. For each $n\in\Z_{\geq0}$ let $R_{\alpha, n}$ be a space of polynomials defined by
\begin{align*}
R_{\alpha, n}\coloneqq\{f(x)\in\C[x_1,\dots, x_d]\mid \text{$f(gx)=\chi_\alpha(g)^nf(x)$ for all $g\in K_\C$}\}.
\end{align*}
 Then $R_\alpha$ is naturally identified with $\bigoplus_{n\geq0}R_{\alpha, n}$.
\begin{defi}
Define $\C^d\sslash_\alpha K_\C\coloneqq{\rm Proj}(\bigoplus_{n\geq0}R_{\alpha, n})$. This is called the {\it geometric invariant theory (GIT) quotient.}
\end{defi}
\begin{defi}
We say $z\in\C^d$ is {\it $\alpha$-semistable} if there exists an $f(x)\in R_{\alpha, n}$ with $n\in\Z_{>0}$ such that $f(z)\neq0$. We denote by $(\C^d)^{\alpha-ss}$ the set of all $\alpha$-semistable points.
\end{defi}
We refer the reader to \cite{Mum1} for a proof of the following Proposition \ref{Zariski closed}:
\begin{prop}\label{Zariski closed}
{\it 
Let $V$ be a complex vector space and $G\subseteq \GL(V)$ an algebraic subgroup.  Then we have the following:
\begin{align*}
\overline{\overline{G\cdot p}}=\overline{G\cdot p} \ \ \text{for all $p\in V$},
\end{align*}
where we denote by $\overline{\overline{G\cdot p}}$ the Euclidean closure, and by $\overline{G\cdot p}$ the Zariski closure. In particular, $G\cdot p$ is closed with respect to the Euclidean topology if and only if it is closed with respect to the Zariski topology.
}
\end{prop}
The GIT quotient can be described as follows:
\begin{prop}\label{morphism}
{\it There exists a categorical quotient $\phi:(\C^d)^{\alpha-ss}\rightarrow \C^d\sslash_\alpha K_\C$ which satisfies the following properties. For each $z, z^\prime\in(\C^d)^{\alpha-ss}$, $\phi(z)=\phi(z^\prime)$ holds if and only if $\overline{K_\C\cdot z}\cap\overline{K_\C\cdot z^\prime}\cap(\C^d)^{\alpha-ss}\neq\emptyset$ and further for each $q\in \C^d \sslash_\alpha K_\C$, $\phi^{-1}(q)$ contains a unique $K_\C$-orbit which is closed in $(\C^d)^{\alpha-ss}$. 
}
\end{prop}
\begin{proof}
See \cite{Dol1}, \cite{MFK1} and \cite{New1}.
\end{proof}
We define an equivalence relation $\sim$ on $(\C^d)^{\alpha-ss}$ as follows:
\begin{align*}
z\sim z^\prime \Longleftrightarrow\overline{K_\C\cdot z}\cap\overline{K_\C\cdot z^\prime}\cap(\C^d)^{\alpha-ss}\neq\emptyset \ \ \text{for $z, z^\prime\in(\C^d)^{\alpha-ss}$}.
\end{align*}
Then by Proposition \ref{morphism}, $\C^d\sslash_\alpha K_\C$ is identified with $(\C^d)^{\alpha-ss}/\sim$. Moreover for each equivalent class there exists a $z\in(\C^d)^{\alpha-ss}$ such that $K_\C\cdot z=(\C^d)^{\alpha-ss}\cap \overline{K_\C\cdot z}$ and such a $z$ is unique up to a transformation of $K_\C$. 

$\alpha$-semistable points are characterized as follows:
\begin{prop}\label{semistable}
{\it
The following are equivalent for each $z\in\C^d$:
\begin{enumerate}[(1)]
\item \label{semi1} $z$ is $\alpha$-semistable;
\item \label{semi2} $\alpha$ satisfies the following:
\begin{align*}
\alpha\in\sum_{j\in J_z}\Q_{\geq0}(\iota^\ast u^j/2\pi),
\end{align*}
where $J_z$ denotes $\{j\in\{1, \dots, d\}\mid z_j\neq 0\}$;
\item \label{semi3}
$\alpha$ is in the cone generated by $(\iota^\ast u^j/2\pi)_{j\in J_z}$:
\begin{align*}
\alpha\in\sum_{j\in J_z}\R_{\geq0}(\iota^\ast u^j/2\pi);
\end{align*}
\item \label{semi4}For each $v\in \C\backslash\{0\}$, $\overline{K_\C\cdot(z, v)}$ does not intersect with $\C^d\times \{0\}$.
\end{enumerate}
}
\end{prop}
\begin{proof}
$(\ref{semi1})\Leftrightarrow(\ref{semi2})$ This can be proved by the same argument as in the proof of \cite[Lemma 3.4]{Kon1}. 

$(\ref{semi2})\Leftrightarrow (\ref{semi3})$ This follows from the general theory of polyhedral convex cones. See \cite{Ful1}. 

 $(\ref{semi1})\Rightarrow (\ref{semi4})$ Suppose $z$ is $\alpha$-semistable. We take an $f\in R_{n,\alpha}$ such that $n\in \Z_{>0}$ and $f(z)\neq 0$. We define a polynomial $\hat{f}(x, y)$ by $\hat{f}(x, y)\coloneqq y^nf(x)$. Then we have the following:
 \begin{align*}
 \hat{f}(x, y)|_{\overline{K_\C\cdot(z, v)}}&\equiv v^nf(z), \\
 \hat{f}(x, y)|_{\C^d\times\{0\}}&\equiv 0,
 \end{align*}
 and thus (\ref{semi4}) holds.
 
 $(\ref{semi4})\Rightarrow (\ref{semi1})$ Suppose (\ref{semi4}) holds. Then there exists a polynomial $\hat{f}(x, y)$ such that
 \begin{align*}
 \hat{f}(x, y)|_{\overline{K_\C\cdot(z, v)}}&\equiv 1, \\
 \hat{f}(x, y)|_{\C^d\times\{0\}}&\equiv 0.
 \end{align*}
The polynomial $\hat{f}(x, y)$ can be written as $\hat{f}(x, y)=yf_1(x)+\cdots+y^mf_m(x)$. Take an $n\in\{1, \dots, m\}$ such that $f_n(x)\neq 0$. Then we have $f_n\in R_{n,\alpha}$ and $f_n(z)\neq 0$.
\end{proof}
The closed orbits are characterized as follows:
\begin{prop}\label{closed orbits}
{\it
The following are equivalent for each $z\in\C^d$:
\begin{enumerate}[(1)]
\item \label{1}$z$ is $\alpha$-semistable and the $K_\C$-orbit is closed in $(\C^d)^{\alpha-ss}$: 
\begin{align*}
K_\C\cdot z=\overline{K_\C\cdot z}\cap (\C^d)^{\alpha-ss};
\end{align*}
\item \label{2}
$\alpha$ satisfies the following:
\begin{align*}
\alpha\in\sum_{j\in J_z}\Q_{>0}(\iota^\ast u^j/2\pi);
\end{align*}
\item \label{3}
$\alpha$ is in the interior of the cone generated by $(\iota^\ast u^j/2\pi)_{j\in J_z}$:
\begin{align*}
\alpha\in\sum_{j\in J_z}\R_{>0}(\iota^\ast u^j/2\pi);
\end{align*}
\item \label{4}
The following holds:
\begin{align*}
\sum_{j\in J_z}\R(\iota^\ast u^j/2\pi)=\sum_{j\in J_z}\R_{\geq0}(\iota^\ast u^j/2\pi)+\R_{\geq0}(-\alpha);\end{align*}
\item \label{5}For each $v\in \C\backslash\{0\}$, $K_\C\cdot (z, v)$ is closed;
\item \label{6}
The following holds:
\begin{align*}
\mu_K^{-1}(-\alpha)\cap K_\C\cdot z\neq \emptyset.
\end{align*}
\end{enumerate}
}
\end{prop}
\begin{proof}
$(\ref{1})\Rightarrow (\ref{5})$ Suppose (1) holds. By the general theory of algebraic groups, there uniquely exists a closed orbit which is contained in $\overline{K_\C\cdot(z,v)}$. Let $K_\C\cdot(z^\prime, v)$ be such a closed orbit. Then by Proposition \ref{semistable}, $z^\prime\in(\C^d)^{\alpha-ss}$. We take a sequence $(g_i)_{i\in\N}$ such that 
\begin{align*}
(z^\prime, v)=\lim_{i\to\infty}g_i\cdot(z,v).
\end{align*}
Therefore we have $z^\prime=\lim_{i\to\infty}g_i\cdot z$, and thus we see $z^\prime\in\overline{K_\C\cdot z}\cap (\C^d)^{\alpha-ss}$. Then (\ref{5}) holds.

$(\ref{5})\Rightarrow (\ref{1})$ Suppose (\ref{5}) holds. Let  $z^\prime\in\overline{K_\C\cdot z}\backslash K_\C\cdot z$. We take a sequence $(g_i)_{i\in\N}$ so that $z^\prime=\lim_{i\to\infty}g_i\cdot z$. Since $K_\C\cdot (z,1)$ is closed, we have $\lim_{i\to\infty}|\chi_\alpha(g_i)^{-1}|=\infty$. This implies that $\lim_{i\to\infty}(g_i^{-1}z^\prime, \chi_{\alpha}(g_i))\in\C^d\times\{0\}$ and thus we have $z^\prime\notin(\C^d)^{\alpha-ss}$.

$(\ref{2})\Leftrightarrow (\ref{3})\Leftrightarrow (\ref{4})$ This follows from the general theory of polyhedral convex cones. See \cite{Ful1}. 

$(\ref{3})\Leftrightarrow(\ref{6})$ We shall prove this in Proposition \ref{equivalent}. 

$(\ref{4})\Leftrightarrow (\ref{5})$ See \cite[pp.30-31]{Nak1}.
\end{proof}
 The equivalence of (\ref{2}) and (\ref{3}) holds for any $\lambda\in k^\ast$:

\begin{prop}\label{equivalent}
{\it
Let $\lambda\in k^\ast$ and $z\in\C^d$. We define a functional $l_{\lambda, z}:k\rightarrow \R$ by
\begin{align*}
l_{\lambda, z}(v)\coloneqq\frac{1}{4}\sum_{j=1}^d|z_j|^2e^{2\langle u^j, v\rangle}-\langle\lambda, v\rangle.
\end{align*}
Then the following are equivalent:
\begin{enumerate}[(1)]
\item \label{equivalent 1}
$\lambda$ is in the interior of the cone generated by $(\iota^\ast u^j/2\pi)_{j\in J_z}$:
\begin{align*}
\lambda\in\sum_{j\in J_z}\R_{>0}\iota^\ast u^j;
\end{align*}
\item 
\label{equivalent 2}
The following holds:
\begin{align*}
\mu_K^{-1}(-\lambda)\cap K_\C\cdot z\neq \emptyset;
\end{align*}
\item \label{equivalent 3}$l_{\lambda, z}$ attains a minimum.
\end{enumerate}
Moreover if  $v$ and $v^\prime$ be minimizers of $l_{\lambda, z}$, then $v-v^\prime$ is in the orthogonal complement of $\sum_{j\in J_z}\R\iota^\ast u^j$.
}
\end{prop}

\begin{proof}
We assume that $(\iota^\ast u^j)_{j\in J_z}$ generates $k^\ast$ for simplicity. 
Then a direct computation shows that $l_{\lambda, z}$ is strictly convex.  We also see that for each $v\in k$,  $v$ is a critical point of $l_{\lambda, z}$ if and only if the following holds.
\begin{align*}
\lambda=\frac{1}{2}\sum_{j=1}^de^{2\langle u^j, v\rangle}|z_j|^2\iota^\ast u^j.
\end{align*}
Therefore (\ref{equivalent 2}) and (\ref{equivalent 3}) are equivalent. Clearly, (\ref{equivalent 2}) implies (\ref{equivalent 1}). Assume that (\ref{equivalent 1}) holds. We show that (\ref{equivalent 3}) holds. From the assumption, there exists a positive numbers $(s_j)_{j\in J_z}$ such that $\lambda=\sum_{j\in J_z}s_j\iota^\ast u^j.$ Then the functional $l_{\lambda,z}$ is denoted as
\begin{align*}
l_{\lambda,z}(v)=\sum_{j\in J_z}(|z_j|^2e^{2\langle u^j,v\rangle}-s_j\langle u^j,v\rangle).
\end{align*}
This implies that $\lim_{t\to\infty}l_{\lambda,z}(tv)=\infty$ for each $v\neq0$ and thus the functional $l_{\lambda, z}$ attains a minimum.
\end{proof}
From Proposition \ref{morphism}, Proposition \ref{closed orbits} and Proposition \ref{equivalent}, we have the following:
\begin{cor}
{\it The following map is bijective:
\begin{align*}
\mu_K^{-1}(-\alpha)/K\longrightarrow(\C^d)^{\alpha-ss}/\sim.
\end{align*}
}
\end{cor}

\end{document}